\documentclass[11pt]{amsart}

\usepackage{amssymb,bbm,amscd,mathrsfs}
\usepackage{graphicx}
\usepackage{a4wide}
\usepackage{url}

\theoremstyle{plain}
\newtheorem{theorem}{Theorem}[section]
\newtheorem{prop}[theorem]{Proposition}
\newtheorem{lemma}[theorem]{Lemma}
\newtheorem{coro}[theorem]{Corollary}

\theoremstyle{definition}
\newtheorem{remark}{Remark}
\newtheorem{definition}{Definition}
\newtheorem{example}{Example}

\numberwithin{equation}{section}

\newcommand{\ts}{\hspace{0.5pt}}

\DeclareMathOperator{\dens}{\mathrm{dens}}

\newcommand{\ZZ}{\mathbb{Z}}
\newcommand{\RR}{\mathbb{R}\ts}

\newcommand{\NN}{\mathbb{N}}
\newcommand{\QQ}{\mathbb{Q}}

\begin{document}

\title[Densities of $\mathscr A$-free sets of ideals]{On the logarithmic
  probability that a random
  integral ideal is $\mathscr A$-free}

\author{Christian Huck}

\address{Fakult\"{a}t f\"{u}r Mathematik, Universit\"{a}t
  Bielefeld,\newline \hspace*{\parindent}Postfach 100131, 33501
  Bielefeld, Germany} \email{huck@math.uni-bielefeld.de}


\begin{abstract}
  This extends a theorem of Davenport and Erd\"os~\cite{ED} on
  sequences of rational integers to sequences of integral ideals in
  arbitrary number fields $K$. More precisely, we introduce a
  logarithmic density for sets of integral ideals in $K$ and provide a
  formula for the logarithmic density of the set of so-called
  $\mathscr A$-free ideals, i.e.\ integral ideals that are not
  multiples of any ideal from a fixed set $\mathscr A$.
\end{abstract}

\maketitle
\thispagestyle{empty}

\section{Introduction}
Recently, the dynamical and spectral properties of so-called
$\mathscr A$-free systems as given by the orbit closure of the
square-free integers, visible lattice points and various
number-theoretic generalisations have received increased attention;
see~\cite{BH,BKKL,CV,ELD} and references therein. One reason is the
connection of one-dimensional examples such as the square-free
integers with Sarnak's conjecture ~\cite{Sarnak} on the `randomness'
of the M\"obius function, another the explicit computability of
correlation functions as well as eigenfunctions for these systems
together with intrinsic ergodicity properties. Here, we provide a very
first step towards the study of a rather general notion of freeness
for sets of integral ideals in an algebraic number field $K$.

A well known result by Benkoski~\cite{Benkoski} states that the
probability that a randomly chosen $m$-tuple of integers is relatively
$l$-free (the integers are not divisible by a common nontrivial $l$th
power) is $1/\zeta(lm)$, where $\zeta$ is the Riemann zeta
function. In a recent paper Sittinger~\cite{Sittinger} reproved that
formula and gave an extension to arbitrary rings of algebraic integers
in number fields $K$. Due to a lack of unique prime factorisation of
integers in this general situation, one certainly passes to counting
integral ideals as a whole and, with a natural notion of asymptotic
density, the outcome is $1/\zeta_K(lm)$, where
\[
\zeta_K(s)=\sum_{0\neq \mathfrak a\subset \mathcal O_K}\frac{1}{N(\mathfrak a)^s}
\] 
is the Dedekind zeta function of $K$. This immediately leads to the
question if the result allows for a further generalisation to more
general notions of freeness, where one forbids common divisors from an
arbitrary set $\mathscr A$ of non-zero integral ideals instead of
considering merely the set consisting of all prime-powers of the form
$\mathfrak p^l$ with $\mathfrak p\subset \mathcal O_K$ prime. In the
special case $K=\QQ$ and $m=1$, this was successfully done in a paper
by Davenport and Erd\"os~\cite{ED} from 1951. The goal of this short
note is to provide a full generalisation of their result to arbitrary
rings of algebraic integers. It turns out that, building on old and
new results from analytic number theory, one can easily adjust their
argument to the more general situation. In this generality, the case
$m\ge 2$ remains open.

\section{Preliminaries}
Let $K$ be a fixed algebraic number field of degree $d=[K:\QQ]\in\NN$. Let
$\mathcal O_K$ denote the ring of integers of $K$ and recall that
$\mathcal O_K$ is a Dedekind domain~\cite{Neukirch}. Hence we have unique
factorisation of non-zero ideals into prime ideals at our
disposal, i.e.\ any non-zero integral ideal $\mathfrak a\subset
\mathcal O_K$ has a (up to
rearrangement) unique
representation of the form
\[
\mathfrak a=\mathfrak p_1\cdot \ldots\cdot \mathfrak p_l\,,
\]
where the $\mathfrak p_i$ are prime ideals. Recall that the (absolute)
norm $N(\mathfrak a)=[\mathcal O_K:\mathfrak a]$ of a non-zero
integral ideal $\mathfrak a\subset \mathcal O_K$ is always
finite. Moreover, the norm is completely multiplicative, i.e.\ one
always has 
$N(\mathfrak a\mathfrak b)=N(\mathfrak a)N(\mathfrak b)$. A proof of
the following fundamental result can be found in~\cite{Marcus}.

\begin{prop}\label{idealcount}
Let $H(x)$ be the number of non-zero integral ideals with norm less than or equal to $x$. Then 
\[
H(x)=cx+O(x^{1-\frac{1}{d}})
\]
for some positive constant $c$.
\end{prop}

\begin{coro}\label{asympden}
As $x\to\infty$, one has
\[
\sum_{N(\mathfrak a)\le x}\frac{1}{N(\mathfrak
  a)}\,\,\sim \,\,c\log x\,,
\]
where $c$ is the constant from Proposition~\ref{idealcount}.
\end{coro}
\begin{proof}
For $k\in\NN$,  let $h(k)$ denote the number of
non-zero  integral ideals with norm equal to $k$.  Summation by parts yields
\begin{eqnarray*}
\sum_{N(\mathfrak a)\le x}\frac{1}{N(\mathfrak
  a)}&=&\sum_{k=1}^{\lfloor x\rfloor}\frac{h(k)}{k}\\
&=& \frac{H(\lfloor
                                               x\rfloor)}{\lfloor
                                               x\rfloor}+\sum_{k=1}^{\lfloor
                                               x\rfloor-1}\frac{H(k)}{k(k+1)}\\
&=& c+O(x^{-\frac{1}{d}})+c\sum_{k=1}^{\lfloor
    x\rfloor-1}\frac{1}{k+1}+O\Big(\sum_{k=1}^{\lfloor
    x\rfloor-1}\frac{k^{-\frac{1}{d}}}{k+1}\Big)\\
&=& c\sum_{k=1}^{\lfloor
    x\rfloor-1}\frac{1}{k+1} +O(1)\\
&\sim& c\log x\,,
\end{eqnarray*}
since $\sum_{k=1}^{\lfloor
    x\rfloor}\frac{1}{k}\sim \log x$ as $x\to\infty$.
\end{proof}

The following generalisation of Mertens' third  theorem to partial Euler products of the Dedekind
zeta function $\zeta_K(s)$ of $K$ at $s=1$ was shown by Rosen. It will
turn out to be crucial for our main result.

\begin{theorem}\cite{Rosen}\label{Mertens}
There is a positive constant $C$ such that
\[
\prod_{N(\mathfrak p)\le x}\Big(1-\frac{1}{N(\mathfrak
  p)}\Big)^{-1}=C\log x +O(1)\,,
\]
where $\mathfrak p$ ranges over the prime ideals of $\mathcal O_K$. In particular,
$\prod_{N(\mathfrak
  p)\le x}(1-\frac{1}{N(\mathfrak p)})^{-1}\sim C\log
x$ as $x\to\infty$. 
\end{theorem}

\begin{remark}
  In fact, Rosen shows that the constant $C$ above is given by
  $C=\alpha_Ke^{\gamma}$, where $\alpha_K$ is the residue of
  $\zeta_K(s)$ at $s=1$ and $\gamma$ is the Euler-Mascheroni constant.
\end{remark}

Let
$\mathscr A=\{\mathfrak a_1,\mathfrak a_2,\dots\}$ be a fixed set of
non-zero integral ideals $\mathfrak a_i\subset \mathcal O_K$. We are
interested in the set 
\[
\mathcal M_{\mathscr A}:=\{\mathfrak b\neq 0 \mid \exists i\,\, \mathfrak b \subset \mathfrak a_i\}
\]
of non-zero integral ideals that are 
multiples of some $\mathfrak a_i$ respectively its complement in the
set of all non-zero integral ideals
\[
\mathcal V_{\mathscr A}:=\{\mathfrak b\mid \forall i\,\, \mathfrak b\not\subset \mathfrak a_i\}
\]
of so-called \emph{$\mathscr A$-free}\/ (or \emph{$\mathscr
  A$-prime}\/) integral ideals. More precisely, we ask
 if the natural asymptotic densities of these sets exist. In
 general, one defines densities of sets of non-zero
integral ideals as follows.

\begin{definition}
Let $S$ be a set of  non-zero
integral ideals $\mathfrak b\subset \mathcal O_K$. and let $S(x)$ be
the subset of those $\mathfrak b$ with $N(\mathfrak b)\le x$. 
\begin{itemize}
\item[(1)]
The \emph{upper/lower (asymptotic) density}\/ $D(S)$/$d(S)$ of $S$ is defined as
\[
\limsup_{x\to\infty}\text{/}\liminf_{x\to\infty}\frac{S(x)}{H(x)}\,.
\]
If these numbers coincide, the common value is called the
\emph{(asymptotic) density}\/ of $S$, denoted by $\dens (S)$.
\item[(2)]
The \emph{upper/lower (asymptotic) logarithmic density}\/ $\Delta
(S)$/$\delta(S)$ of $S$ is defined as
\[
\limsup_{x\to\infty}\text{/}\liminf_{x\to\infty}\frac{\sum_{\substack{\mathfrak
      b\in S\\
    N(\mathfrak b_)\le x}}\frac{1}{N(\mathfrak
  b)}}{\sum_{\substack{0 \neq \mathfrak
      b\subset\mathcal O_K\\
    N(\mathfrak b)\le x}}\frac{1}{N(\mathfrak
  b)}}\,,
\]
where one might substitute the denominator by $c\log x$ due to
Corollary~\ref{asympden}. Again, if these numbers coincide, the common value is called the
\emph{(asymptotic) logarithmic density}\/ of $S$, denoted by $\dens_{\log}(S)$.
\end{itemize}
\end{definition}

As in the well known special case of rational integers, the above
lower und upper densities are related as follows.

\begin{lemma}[Density inequality]
For any set $S$
of non-zero
integral ideals of $K$, one has
\[
d(S)\,\le \,\delta(S)\,\le \,\Delta(S)\,\le \,D(S)\,.
\]
In particular, the existence of the density of $S$ implies the existence of
the logarithmic density of $S$.
\end{lemma}
\begin{proof}
The assertion follows from summation by parts as follows. Let us first show that $\Delta(S)\le D(S)$. To this end, let $\varepsilon >0$ and
choose $N\in\NN$ such that $\frac{S(n)}{H(n)}\le D(S)+\varepsilon$ for
all $n\ge N$. For $k\in\NN$,  let $s(k)$ denote the number of
non-zero  integral ideals $\mathfrak a\in S$ with norm equal to $k$.
Summation by parts yields for $n\ge N$
\begin{eqnarray*}
\sum_{k=1}^{n}\frac{s(k)}{k}&=& \frac{S(
                                              n)}{n}+\sum_{k=1}^{n-1}\frac{S(k)}{k(k+1)}\\
&\le& \frac{H(
                                              n)}{n}+\sum_{k=1}^{N-1}\frac{S(k)}{k(k+1)}+\sum_{k=N}^{n-1}\frac{S(k)}{k(k+1)}\\
&\le& \frac{H(
                                              n)}{n}+\sum_{k=1}^{N-1}\frac{S(k)}{k(k+1)}+(D(S)+\varepsilon)\sum_{k=N}^{n-1}\frac{H(k)}{k(k+1)}\,.
\end{eqnarray*}
Since $\frac{H(n)}{n}\rightarrow c$ and
$\sum_{k=N}^{n-1}\frac{H(k)}{k(k+1)}\sim c\log n$ as $n\to \infty$
(see the proof of Corollary~\ref{asympden}),
one obtains $\Delta(S)\le D(S)+\varepsilon$. The assertion follows.

For the left inequality $d(S)\le\delta(S)$, let $\varepsilon >0$ and
choose $N\in\NN$ such that $\frac{S(n)}{H(n)}\ge d(S)-\varepsilon$ for
all $n\ge N$. Again, summation by parts yields for $n\ge N$
 \begin{eqnarray*}
\sum_{k=1}^{n}\frac{s(k)}{k}
&\ge&\sum_{k=N}^{n-1}\frac{S(k)}{k(k+1)}\\
&\ge& (d(S)-\varepsilon)\sum_{k=N}^{n-1}\frac{H(k)}{k(k+1)}\,,
\end{eqnarray*}
which as above implies $\delta(S)\ge d(S)-\varepsilon$ and thus the assertion.
\end{proof}

\section{The Davenport-Erd\"os theorem for number fields}

Next, we shall study the densities of the set $\mathcal
M_{\mathscr A}$. Let us start with the finite case. Note that, for a
finite set $\mathcal J$ of integral ideals, their least common
multiple is just the intersection $\bigcap\mathcal J$. 

\begin{prop}
If $\mathscr A$ is finite, then
the density of $\mathcal M_{\mathscr A}$ exists and is given by
\[
\dens (\mathcal M_{\mathscr A})=\sum_{\varnothing\neq \mathcal J\subset
    \mathscr A}(-1)^{\mid \mathcal J\mid+1} \frac{1}{N\big(\bigcap \mathcal J\big)}
\]
\end{prop}
\begin{proof}
  If $\mathfrak b$ is a
  non-zero integral ideal of norm $N(\mathfrak b)$ and divisible by
  $\mathfrak a$, then there is a unique non-zero integral ideal $\mathfrak a'$
  such that $\mathfrak b=\mathfrak a\mathfrak a'$. In particular,
  $N(\mathfrak a')=N(\mathfrak b)/N(\mathfrak a)$ by the
  multiplicativity of the norm. This provides a bijection from the set
  of multiples of $\mathfrak a$ of norm $n$ to the set of non-zero integral ideals
  of norm $n/N(\mathfrak a)$. Hence, by the inclusion-exclusion
  principle, one has
\[
\frac{S(x)}{H(x)}=\sum_{\varnothing\neq \mathcal J\subset
    \mathscr A}(-1)^{\mid \mathcal J\mid+1} H\Big(\frac{x}{N(\bigcap
                     \mathcal J)}\Big)\Big / H(x)\,.
\] 
Application of Proposition~\ref{idealcount} now yields the assertion. 
\end{proof}

Now let $\mathscr A=\{\mathfrak a_1,\mathfrak a_2,\dots\}$ be (countably)
infinite. Since $\dens(\mathcal M_{\{\mathfrak a_1,\dots ,\mathfrak
  a_r\}})$ is an increasing sequence with upper bound $1$, we may
define
\[
A:=\lim_{r\to\infty}\dens(\mathcal M_{\{\mathfrak a_1,\dots ,\mathfrak
  a_r\}})\,.
\]  
It is then natural to ask if, in general, $A$ is the density of
$\mathcal M_{\mathscr A}$. Already in the special case $K=\QQ$ the
answer is negative in the sense that the natural lower and upper densities may
differ; cf.~\cite{Besicovitch}.

\begin{remark}\label{condition}
Due to $\dens(\mathcal M_{\{\mathfrak a_1,\dots ,\mathfrak
  a_r\}})\le d(\mathcal M_{\mathscr A})$ for all $r\in\NN$, one has $A\le d(\mathcal M_{\mathscr A})$.
\end{remark}

\begin{prop}\label{convcond}
If the series $\sum_{\mathfrak a \in\mathscr A}\frac{1}{N(\mathfrak a)}$
converges, then the density of $\mathcal M_{\mathscr A}$ exists and is
equal to $A$.
\end{prop}
\begin{proof}
For fixed $r\in\NN$, the number of elements of $\mathcal M_{\mathscr A}$ up to norm $n$
\emph{not}\/ divisible by any of $\mathfrak a_1,\dots,\mathfrak a_r$
is at most $\sum_{i=r+1}^{\infty}H(\frac{n}{N(\mathfrak
  a_{i})})$. Hence, the corresponding upper density is at most
$\sum_{i=r+1}^{\infty}\frac{1}{N(\mathfrak a_i)}$ and this converges
to $0$ as $r\to\infty$. It follows that the upper density of $\mathcal
M_{\mathscr A}$ is
\[
\dens(\mathcal M_{\{\mathfrak a_1,\dots ,\mathfrak
  a_r\}})+O\Big(\sum_{i=r+1}^{\infty}\frac{1}{N(\mathfrak a_i)}\Big)\,,
\]
which converges to $A$ as $r\to\infty$. This yields $D(\mathcal
M_{\mathscr A})\le A$ and thus the assertion by Remark~\ref{condition}. 
\end{proof}

\begin{example}
Recall that the Dedekind zeta function $\zeta_K(s)$ converges for all
$s>1$ and has the Euler product expansion
\[
\zeta_K(s)\,=\,\sum_{\mathfrak a\neq 0}\frac{1}{N(\mathfrak a)^s}\,=\,\prod_{\mathfrak p}\Big(1-\frac{1}{N(\mathfrak p)^s}\Big)^{-1}\,.
\] 
It follows that, for $l\ge
2$ fixed and $\mathscr A=\{\mathfrak p^l\mid \mathfrak p \text{ prime}\}$, the density of $\mathcal M_{\mathscr A}$ exists and is
equal to
\[
1-\prod_{\mathfrak p}\Big(1-\frac{1}{N(\mathfrak p)^l}\Big)\,=\,1-\frac{1}{\zeta_K(l)}\,.
\]
In other words, the density of $\mathcal V_{\mathscr A}$ exists and is
equal to 
$\frac{1}{\zeta_K(l)}$, in accordance with~\cite[Thm.\ 4.1]{Sittinger}.
\end{example}

As a preparation of the proof below, we next introduce the so-called \emph{multiplicative
  density}\/ of $\mathcal M_{\mathscr A}$. Let $\{\mathfrak
p_1,\mathfrak p_2,\dots \}$ be the set of all prime ideals of $\mathcal
O_K$, with a numbering that corresponds to increasing order with
respect to the norms, i.e.\ $i\le j$ always implies $N(\mathfrak
p_i)\le N(\mathfrak
p_j)$. For $k\in\NN$ fixed, denote by
$\mathfrak n'$ the general non-zero integral ideal composed entirely
of the prime ideals $\mathfrak
p_1,\dots,\mathfrak p_k$ (a so-called \emph{$\mathfrak
p_1,\dots,\mathfrak p_k$-ideal}\/). Then, one has the convergence
\[
\sum_{\mathfrak n'}\frac{1}{N(\mathfrak
  n')}=\prod_{i=1}^{k}\Big(1-\frac{1}{N(\mathfrak p_i)}\Big)^{-1}=: \Pi_k\,.
\] 
Further, denote by $\mathfrak b'$ those ideals from $\mathcal
M_{\mathscr A}$ that are $\mathfrak
p_1,\dots,\mathfrak p_k$-ideals and let
\[
B_k:=\frac{\sum_{\mathfrak b'}\frac{1}{N(\mathfrak
    b')}}{\sum_{\mathfrak n'} \frac{1}{N(\mathfrak n')}}={\Pi_k}^{-1}\sum_{\mathfrak b'}\frac{1}{N(\mathfrak b')}\,.
\]  
If the sequence $B_k$ converges as $k\to \infty$, the limit is called
the \emph{multiplicative density}\/ of $\mathcal M_{\mathscr A}$. Let
$\mathscr A':=\{\mathfrak a_1',\mathfrak a_2',\dots \}$ be the subset of $\mathscr
A$ consisting of the $\mathfrak
p_1,\dots,\mathfrak p_k$-ideals only. Then the $\mathfrak b'$ from above
are precisely those of the form $\mathfrak a_i'\mathfrak n'$. It
follows from the inclusion-exclusion principle and Proposition~\ref{convcond} in conjunction with
the
convergence of $\sum_{\mathfrak a'\in\mathscr A'}\frac{1}{N(\mathfrak
  a')}$ that
\begin{eqnarray*}
\sum_{\mathfrak b'}\frac{1}{N(\mathfrak b')}&=&\sum_{\mathfrak n'}\frac{1}{N(\mathfrak n')}\sum_{\varnothing\neq \mathcal J\subset
    \mathscr A'}(-1)^{\mid \mathcal J\mid+1} \frac{1}{N\big(\bigcap
                                                \mathcal J\big)}\\
&=&\Pi_k\,\,\dens (\mathcal M_{\mathscr A'})\,.
\end{eqnarray*}
One obtains that $B_k=\dens (\mathcal M_{\mathscr A'})$ which shows
that the 
$B_k$ increase with $k$. Since the $B_k$ are bounded above by $1$, this
proves that the $B_k$ indeed converge, say $\lim_{k\to\infty}B_k=:B$.  

Next, we shall show that $B=A$. Clearly, if $k$ is sufficiently large
in relation to $r$, then $\{\mathfrak a_1,\dots,\mathfrak a_r\}\subset
\mathscr A'$. Hence, one has
\[
B\,\ge\, B_k\,=\,\dens (\mathcal M_{\mathscr A'})\,\ge \,\dens(\mathcal M_{\{\mathfrak a_1,\dots ,\mathfrak
  a_r\}})
\]
and therefore $B\ge A$. For the reverse inequality $A\ge B$, let $k$
be fixed. The
convergence of $\sum_{\mathfrak a'\in\mathscr A'}\frac{1}{N(\mathfrak
  a')}$ implies that the density of $\mathcal M_{\mathscr A'}$ exists
and satisfies (see the proof of Propsosition~\ref{convcond}) 
\[
\dens (\mathcal M_{\mathscr A'})\,\le\, \dens(\mathcal M_{\{\mathfrak a_1',\dots ,\mathfrak
  a_r'\}})+\sum_{i=r+1}^{\infty}\frac{1}{N(\mathfrak a_i')}\,.
\]
Now choose $s$ large enough such that $\{\mathfrak a_1',\dots ,\mathfrak
  a_r'\}\subset \{\mathfrak a_1,\dots ,\mathfrak
  a_s\}$. It follows that
\[
\dens(\mathcal M_{\{\mathfrak a_1',\dots ,\mathfrak
  a_r'\}})\,\le \,\dens(\mathcal M_{\{\mathfrak a_1,\dots ,\mathfrak
  a_s\}})\,\le \, A
\]
and further, by letting $r\to\infty$, $\dens (\mathcal M_{\mathscr
  A'})\le A$, i.e.\ $B_k\le A$. It follows that $B\le
A$. Altogether, this proves
the claim $B=A$. We are now in a position to proof the main result of
this short note.
 
\begin{theorem}
The logarithmic density of $\mathcal M_{\mathscr A}$
exists and is equal to $A$. The number $A$ also equals the lower
density of $\mathcal M_{\mathscr A}$.
\end{theorem}
\begin{proof}
We have to show, for $S=\mathcal M_{\mathscr A}$, the equality $d(S)=\delta(S)=\Delta(S)=A$, i.e.\
$\Delta(S)\le A$ or, equivalently, $\Delta(S)\le B$ since we have
already seen
above that $A=B$. Let $k\in\NN$ be fixed. Divide the $\mathfrak b'$
from above of norm $\le x$ into two
classes, placing in the first class those from $\mathcal M_{\mathscr
  A'}$ and in the second class the remaining ones. The $\mathfrak b'$
in the first class have density $B_k$ (see above), hence the sum
$\beta_1(x)$ corresponding to the $\mathfrak b'$ in the first class
satisfies (the density inequality is an equality in this case)
\[
\lim_{x\to\infty}\frac{\beta_1(x)}{c\log x}=B_k\,.
\]
For the sum $\beta_2(x)$ corresponding to the $\mathfrak b'$ in the
second class, let $\{\mathfrak p_1,\dots,\mathfrak p_h\}$ be the set
of all prime ideals with norm up to $x$. The $\mathfrak b'$ in the
second class are $\mathfrak
p_1,\dots,\mathfrak p_h$-ideals, but are not in $\mathcal M_{\mathscr
  A'}$. Denoting by $\mathfrak b^*$ the $\mathfrak b'$ of this kind
(wether of norm $\le x$ or not), one has
\[
\beta_2(x)\le \sum_{\mathfrak b^*}\frac{1}{N(\mathfrak b^*)}\,.
\] 
The $\mathfrak b^*$ are obtained by taking all $\mathfrak
p_1,\dots,\mathfrak p_h$-ideals $\mathfrak b''$, and removing from
them all $\mathfrak b'\mathfrak c$, where $\mathfrak b'$ is a $\mathfrak
p_1,\dots,\mathfrak p_k$-ideal and $\mathfrak c$ is any $\mathfrak
p_{k+1},\dots,\mathfrak p_h$-ideal. Hence
\[
\sum_{\mathfrak b^*}\frac{1}{N(\mathfrak b^*)}\,=\,\sum_{\mathfrak
  b''}\frac{1}{N(\mathfrak b'')}- \sum_{\mathfrak
  b'}\frac{1}{N(\mathfrak b' )}\sum_{\mathfrak c}\frac{1}{N(\mathfrak c)}\,=\,\Pi_hB_h-\Pi_kB_k\sum_{\mathfrak c}\frac{1}{N(\mathfrak c)}\,.
\]
Since
\[
\sum_{\mathfrak c}\frac{1}{N(\mathfrak c)}\,=\,\prod_{i=k+1}^{h}\Big(1-\frac{1}{N(\mathfrak
  p_i)}\Big)^{-1}\,=\,\Pi_h\Pi_k^{-1}\,,
\]
this shows that
\[
\sum_{\mathfrak b^*}\frac{1}{N(\mathfrak b^*)}\,=\,\Pi_h(B_h-B_k)\,.
\]
Finally, it follows from the Mertens type Theorem~\ref{Mertens} by Rosen that
\[
\beta_2(x)\,\le \,\sum_{\mathfrak b^*}\frac{1}{N(\mathfrak
  b^*)}\,=\,\Pi_h(B_h-B_k)\,\le \, C\log x (B_h-B_k)
\] 
and thus, with $\beta(x):=\beta_1(x)+\beta_2(x)$,
\[
\limsup_{x\to\infty}\frac{\beta(x)}{c\log x}\,\le \, B_k+\frac{C}{c}(B-B_k)\,,
\]
since $x\to\infty$ implies $h\to\infty$ which in turn implies $B_h\to
B$. Letting $k\to\infty$ and thus $B_k\to B$, one obtains that
$\Delta(\mathcal M_{\mathscr A})\le B$.
\end{proof}

\begin{coro}
The logarithmic density of $\mathcal V_{\mathscr A}$
exists and is equal to $1-A$. This number also equals the upper
density of $\mathcal V_{\mathscr A}$.
\end{coro}
\begin{proof}
In general, one has $d(S)=1-D(S^c)$ and $\delta(S)=1-\Delta(S^c)$.
\end{proof}

\begin{remark}
  It is natural to ask for an extension of the above results to the
  case of $m$-tuples $(\mathfrak b_1,\dots,\mathfrak b_m)$ of non-zero
  integral ideals, where one studies the set of those tuples that
  consist of simultaneous multiples of ideals from $\mathscr A$
  respectively its complement consisting of the relatively
  $\mathscr A$-free tuples. This is work in progress.
\end{remark}

\begin{remark}
  There is a non-canonical possibility of defining upper and lower
  (asymptotic) densities of sets $S$
  of non-zero integral ideals $\mathfrak b\subset\mathcal O_K$ by
  passing from $S$ to the subset 
\[
\tilde S\,:=\,\{a\in\mathcal O_K\mid
(a)\in S\}
\]
of $\mathcal O_K$ and considering the image
$\alpha(\tilde S)\subset\ZZ^d$ under any isomorphism
$\alpha:\!\,\mathcal O_K\rightarrow \ZZ^d$ of Abelian groups (recall
that $d=[K:\QQ]$). The set $\alpha(\tilde S)$ then has natural upper
and lower densities defined by counting points e.g.\ in centred balls
(or cubes) of radius $R$ in $\RR^d$ divided by the volume and then
considering the $\limsup$ resp.\ $\liminf$ as $R\to\infty$. Note that
this also extends componentwise to the case of $m$-tuples mentioned in
the last remark. In general, it is not clear if the outcome is
independent of the embedding $\alpha$ or coincides with the
corresponding densities introduced above. However, for the set of coprime
$m$-tuples $(\mathfrak b_1,\dots,\mathfrak b_m)$ of non-zero integral
ideals (i.e.\ $\mathfrak b_1+\ldots+\mathfrak b_m=\mathcal O_K$)
resp.\ the set of $m$-tuples $(a_1,\dots,a_m)\in \mathcal O_K^m$ with
$(a_1)+\ldots+(a_m)=\mathcal O_K$, even the (suitably defined)
densities exist and all answers are affirmative (with both densities
equal to $1/\zeta_K(m)$) as follows from~\cite{FM,Sittinger}. Another
coincidence of the two ways of computing densities shows up (with both
densities equal to $1/\zeta_K(l)$) in the case of $l$-free non-zero
integral ideals (non-divisibility by any nontrivial $l$th power)
resp.\ integers in $\mathcal O_K$~\cite{CV,Sittinger}. Proving such a
coincidence in our setting above for the lower density of
$\mathcal M_{\mathscr A}$ remains open, even for the case $m=1$.
\end{remark}

\section*{Acknowledgements}
It is a pleasure to thank Joanna Ku\l
aga-Przymus and Jeanine Van Order for helpful
discussions. This work was initiated during a research in pairs stay
at CIRM (Luminy) in 2016, within the Jean Morlet Chair. It
was supported by the German Research Council (DFG), within the
CRC 701.

\end{document}